\documentclass{amsart}

\usepackage[active]{srcltx}

\usepackage{graphicx}

\usepackage{latexsym}
\usepackage{amsmath,amsthm}
\usepackage{amsfonts}
\usepackage[psamsfonts]{amssymb}
\usepackage{enumerate}

\usepackage[final]
{showkeys}
\usepackage{url}

\RequirePackage[OT1]{fontenc}

\usepackage{calrsfs}

\usepackage{tikz}
\usetikzlibrary{decorations}
\usetikzlibrary{decorations.pathreplacing}

\usepackage{hyperref}

\theoremstyle{plain} 
\newtheorem{theorem}{Theorem}[section]
\newtheorem{corollary}[theorem]{Corollary}
\newtheorem{lemma}[theorem]{Lemma}
\newtheorem{proposition}[theorem]{Proposition}

\newtheorem*{conjecture*}{Conjecture}
\theoremstyle{definition} 

\theoremstyle{definition} 

\theoremstyle{remark} 

\theoremstyle{remark} 

\newtheorem*{remark*}{Remark}
\numberwithin{equation}{section}

\makeatletter
  \renewcommand\section{\@startsection {section}{1}{\z@}%
                                   {-\bigskipamount}%
                                   {\medskipamount}%
                                   {\large\bfseries
                                   \raggedright}}

  \renewcommand\subsection{\@startsection {subsection}{2}{\z@}%
                                   {-\medskipamount}%
                                   {\smallskipamount}%
                                   {\bfseries
                                   \raggedright}}
\makeatother

\newcommand{\limD}{\operatorname{Dlim}}
\renewcommand{\limD}{\mathop{\operatorname{Dlim}}}
\newcommand{\limsupD}{\mathop{\operatorname{Dlimsup}}}
\newcommand{\liminfD}{\mathop{\operatorname{Dliminf}}}

\newcommand{\uperp}{\{u\}^\perp}

\newcommand{\ffrown}{\text{\raisebox{3pt}[0pt][0pt]{$\frown$}}}
\renewcommand{\O}{\underset{\ffrown}{<}}

\renewcommand{\gg}{>\kern-2pt>}
\renewcommand{\ll}{<\kern-2pt<}

\renewcommand{\bar}{\overline}

\renewcommand{\S}{\operatorname{\mathsf{S}\!}}
\renewcommand{\S}{\mathsf{S}}

\newcommand{\V}{\mathsf{V}}

\renewcommand{\gg}{>\kern-2pt>}
\renewcommand{\ll}{<\kern-2pt<}

\newcommand{\dd}{\partial}
\renewcommand{\dd}{\operatorname{d}\!}

\renewcommand{\le}{\leqslant}
\renewcommand{\ge}{\geqslant}

\newcommand{\al}{\alpha}

\newcommand{\Ga}{\Gamma}
\newcommand{\ep}{\varepsilon}

\renewcommand{\th}{\theta}

\newcommand{\la}{\lambda}

\newcommand{\de}{\delta}

\newcommand{\vpi}{\varphi}

\newcommand{\B}{\mathfrak{B}}

\newcommand{\D}{\mathcal{D}}
\newcommand{\EE}{\mathcal{E}}

\newcommand{\Y}{\mathcal{Y}}

\newcommand{\LD}{\mathcal{L}\!\mathcal{D}}
\renewcommand{\LD}{\mathcal{L}{\kern -1.9pt}\mathcal{D}}
\renewcommand{\LD}{\mathcal{D}}
\renewcommand{\LD}{\mathcal{L}{\kern -4pt}\mathcal{C}}
\renewcommand{\LD}{\mathcal{R}{\kern -3pt}\mathcal{C}}

\newcommand{\ii}[1]{\mathrm{I}\!\left\{#1\right\}}

\renewcommand{\P}{\operatorname{\mathsf{P}}} 

\newcommand{\E}{\operatorname{\mathsf{E}}}

\newcommand{\R}{{\mathbb{R}}}
\newcommand{\C}{\mathbb{C}}

\newcommand{\tP}{{\tilde{P}}}
\newcommand{\txi}{{\tilde{\xi}}}
\newcommand{\tQ}{{\tilde{Q}}}
\newcommand{\tA}{{\tilde{A}}}
\newcommand{\tB}{{\tilde{B}}}
\newcommand{\teta}{{\tilde{\eta}}}

\newcommand{\A}{\mathcal{A}}

\newcommand{\inter}{\mathrm{int}\,}

\begin{document}


\title{Quantifying 
minimal non-collinearity among random points}


\author{Iosif Pinelis}

\address{Department of Mathematical Sciences\\
Michigan Technological University\\
Hough\-ton, Michigan 49931, USA\\
E-mail: ipinelis@mtu.edu}



%
%
%
%
%


\begin{abstract}
Let $\varphi_{n,K}$ denote the largest angle in all the triangles with vertices among the $n$ points selected at random in a compact convex subset $K$ of $\mathbb{R}^d$ with nonempty interior, where $d\ge2$.  
It is shown that the distribution of the random variable $\lambda_d(K)\,\frac{n^3}{3!}\,(\pi-\varphi_{n,K})^{d-1}$, where $\lambda_d(K)$ is a certain positive real number which depends only on the dimension $d$ and the shape of $K$,  
converges to the standard exponential distribution as $n\to\infty$. 
By using the Steiner symmetrization, it is also shown that $\lambda_d(K)$ -- which is referred to in the paper as the elongation of $K$ -- attains its minimum if and only if $K$ is a ball $B^{(d)}$ in $\mathbb{R}^d$. Finally, the asymptotics of $\lambda_d(B^{(d)})$ for large $d$ is determined.  
\end{abstract}

\maketitle

\tableofcontents

\section{
Summary
and discussion
}\label{intro} 

Let $K$ be a bounded convex subset of $\R^d$ with nonempty interior, for some natural $d\ge2$. Let $P_1,P_2,\dots$ be random points drawn independently and uniformly from $K$. For each natural $n$, let $T_n$ be the set of all subsets of the set $\{1,\dots,n\}$ of cardinality $3$. For each $t=\{i,j,k\}\in T_n$, let $X_t$ denote the largest angle (with possible angle values in the interval $[0,\pi]$) in the triangle $P_iP_jP_k$ with vertices $P_i,P_j,P_k$. Let 
\begin{equation}\label{eq:la}
	\la_d(K):=3\rho_d\frac{\E|P_1P_2|^d}{\V_d(K)}, 
\end{equation}
where 
\begin{equation}\label{eq:rho}
	\rho_d:=\frac{\pi^{d/2}\Ga(d)}{2^{2d-1}\Ga((d+1)/2)\Ga(d+1/2)}, 
\end{equation}
$|P_1P_2|:=\|P_1-P_2\|$ is the Euclidean distance between $P_1$ and $P_2$, and $\V_d$ is the $d$-dimensional volume (that is, the Lebesgue measure on $\R^d$). 
Let 
\begin{equation}\label{eq:phi}
	\vpi_{n,K}:=\max_{t\in T_n}X_t, 
\end{equation}
the largest angle in all the triangles with vertices among the points $P_1,\dots,P_n$. 

\begin{theorem}\label{th:}
The random variable (r.v.) 
\begin{equation}\label{eq:Y}
	Y_{n,d}:=\la_d(K)\,\frac{n^3}{3!}\,(\pi-\vpi_{n,K})^{d-1}
\end{equation}
converges as $n\to\infty$ to an exponential r.v.\ in distribution: 
\begin{equation*}
	\P(Y_{n,d}>a)\to e^{-a}
\end{equation*}
for each real $a>0$, and hence uniformly in real $a>0$. 
\end{theorem}

This theorem will be proved in Section~\ref{proof}. 

In the particular case when $d=2$ and $K$ is a square in $\R^2$, a version of Theorem~\ref{th:} with an unspecified constant in place of $\la_d(K)$ was presented on the MathOverflow site \cite{alignment}. 

Note that the constant $\la_d(K)$ is, naturally, invariant with respect to all translations, homotheties, and orthogonal transformations applied to the set $K$. 

The conclusion of Theorem~\ref{th:} can be rewritten as follows: 
\begin{equation*}
	\P(\vpi_{n,K}>c)-\Big(1-\exp\Big\{-\la_d(K)(\pi-c)^{d-1}\frac{n^3}{3!}\Big\}\Big)
	\underset{n\to\infty}\longrightarrow0
\end{equation*}
uniformly in $c\in(0,\pi)$. Since $1-\exp\{-\la_d(K)(\pi-c)^{d-1}\tfrac{n^3}{3!}\}$ is increasing in $\la_d(K)$, we see that, the greater the constant $\la_d(K)$, the greater the largest angle $\vpi_{n,K}$ tends to be, in an asymptotic stochastic sense. 

This provides a reason to refer to the constant $\la_d(K)$ as the \emph{elongation (coefficient)} of the convex set $K$. In particular, if the the convex set $K$ is close to a straight line segment, then the volume $\V_d(K)$ will be relatively small and hence the elongation $\la_d(K)$ will be large; accordingly, the largest angle $\vpi_{n,K}$ will then tend to be close to its maximum possible value, $\pi$. 
Another justification for using this term, elongation, is provided by 

\begin{proposition}\label{prop:B,min elong}
Take any natural $d$. The minimum of the elongation $\la_d(K)$ over all bounded convex subsets $K$ of $\R^d$ with nonempty interior is attained only when the closure of $K$ is a closed ball in $\R^d$ of a positive radius.  
\end{proposition}

Proposition~\ref{prop:B,min elong} will be proved in Appendix~\ref{min elong}. 
The proof is based on the Steiner symmetrization; see e.g.\ \cite{klartag04}. In particular, we shall show that the Steiner symmetrization does not increase the elongation. 

In view of the foregoing discussion, the following appears natural: 

\begin{conjecture*}
The largest angle $\vpi_{n,K}$ becomes stochastically smaller after any Steiner symmetrization of $K$. That is, for any natural $d$, any natural $n\ge3$, any convex subset $K$ of $\R^d$, any unit vector $u$ in $\R^d$, and any $c\in(0,\pi)$, one has 
\begin{equation*}
	\P(\vpi_{n,\S_uK}>c)\le\P(\vpi_{n,K}>c), 
\end{equation*}
where $\S_u$ denotes the Steiner symmetrization along the vector $u$. 
\end{conjecture*}

Of course, one may also define the relative elongation of $K$, by dividing the elongation of $K$ by that of the unit ball. Then Proposition~\ref{prop:B,min elong} can be restated as follows: the relative elongation of $K$ is never less than $1$, and it equals $1$ if and only if the closure of $K$ is a closed ball in $\R^d$ of a positive radius. 

The asymptotics (for $d\to\infty$) of the elongation coefficient of balls in $\R^d$ is presented in 

\begin{proposition}\label{lem:la(B)}
Let $B^{(d)}$ denote the unit ball in $\R^d$. Let $P_1$ and $P_2$ be random points drawn independently and uniformly from $B^{(d)}$. Then 
\begin{equation}\label{eq:la(B)}
	\E|P_1P_2|^d\sim 
		\frac1{\sqrt{6}}\,\Big(\frac{8}{3\sqrt3}\Big)^d\quad\text{and}\quad
	\la_d(B^{(d)})\sim\sqrt3\,\Big(\frac2{3\sqrt3}\Big)^d
	\quad\text{as }d\to\infty.  
\end{equation} 
\end{proposition}

This proposition will be proved in Appendix~\ref{elong}. 

In particular, it follows from Proposition~\ref{lem:la(B)} that the $L_d$ norm 
$\| |P_1P_2|\|_d:=(\E|P_1P_2|^d)^{1/d}$ of the distance $|P_1P_2|$ between the random points $P_1$ and $P_2$ in the unit ball converges (as $d\to\infty$) to $\frac{8}{3\sqrt3}=1.539\dots$, which is strictly less than the diameter $2$ of the unit ball. Mainly, this is a consequence of the fact that the maximum of the function $\psi$ in \eqref{eq:Q_d=} occurs at the point $t_*=-\frac13$.  

Let $\limD$ denote the limit in distribution, for $\limD_{n\to\infty}Y_n\overset{\mathrm{D}}=Y$ to mean that $Y_n\underset{n\to\infty}\longrightarrow Y$ in distribution, where $Y$ and the $Y_n$'s are real-valued r.v.'s. 
Next, let $\limsupD$ denote the limit superior in distribution, for $\limsupD_{n\to\infty}Y_n\overset{\mathrm{D}}\le Y$ to mean that $\limsup_{n\to\infty}\P(Y_n\ge y)\le\P(Y\ge y)$ for all real $y$. Similarly, one may let $\liminfD_{n\to\infty}Y_n\overset{\mathrm{D}}\ge Y$ mean that $\limsupD_{n\to\infty}(-Y_n)\overset{\mathrm{D}}\le -Y$ or, equivalently,
$\liminf_{n\to\infty}\P(Y_n>y)\ge\P(Y>y)$ for all real $y$. 
Then $\limD_{n\to\infty}Y_n\overset{\mathrm{D}}=Y$ if and only if $\limsupD_{n\to\infty}Y_n\overset{\mathrm{D}}\le Y$ and $\liminfD_{n\to\infty}Y_n\overset{\mathrm{D}}\ge Y$. 
In the case when the r.v.\ $Y$ is actually non-random, here one may simply write $\le$,\, $\ge$,\, and $=$\, instead of the symbols $\overset{\mathrm{D}}\le$,\, $\overset{\mathrm{D}}\ge$,\, and $\overset{\mathrm{D}}=$\, (respectively).  
In these terms, one can now state the following immediate corollary of Theorem~\ref{th:} and Propositions~\ref{prop:B,min elong} and \ref{lem:la(B)}. 

\begin{corollary} \label{cor:d to infty}
\begin{equation}\label{eq:d to infty}
\limD\limits_{d\to\infty}\limD_{n\to\infty}\big(n^{3/(d-1)}(\pi-\vpi_{n,B^{(d)}})\big)=\frac{3\sqrt3}2  
\ge\limsupD_{d\to\infty}\limD_{n\to\infty}\big(n^{3/(d-1)}(\pi-\vpi_{n,K_d})\big), 	
\end{equation}
%
where, for each natural $d\ge2$, $K_d$ is an arbitrary bounded convex subset of $\R^d$ with nonempty interior. 
\end{corollary}

Note that the double-limit value/bound $\frac{3\sqrt3}2$ in \eqref{eq:d to infty} is non-random; therefore, the symbol $\mathrm{D}$ is not overset there. 
Somewhat crudely but perhaps more transparently, \eqref{eq:d to infty} may be expressed in terms of an approximate equality $\approx$ and an ``approximate inequality'' $\lessapprox$, as follows: for large $d$ and very large $n$, 
\begin{equation*}
	\vpi_{n,B^{(d)}}\approx\pi-\frac{3\sqrt3/2}{n^{3/(d-1)}}\lessapprox\vpi_{n,K_d},  
\end{equation*}
with non-random approximate value/bound $\pi-\frac{3\sqrt3/2}{n^{3/(d-1)}}$. 

%

\section{Proof of Theorem~\ref{th:}} \label{proof}
Without loss of generality (wlog), (i) the diameter of the set $K$ (defined as \break 
$\sup\{\|x-y\|\colon x,y \text{ in } K\}$) equals $1$ and (ii) the point $0\in\R^d$ is in the interior of $K$, so that $B_0(r)\subseteq K$ for some real $r>0$; here, as usual,  $B_x(r):=\{y\in\R^d\colon\|x-y\|<r\}$ for any $x\in\R^d$. 

\begin{lemma}\label{lem:1}
Take any $\de\in(0,1)$ and any $x\in K_\de$, where 
\begin{equation*}
K_\de:=(1-\de)K:=\{(1-\de)z\colon z\in K\}. 	
\end{equation*}
Then $B_x(r\de)\subseteq K$. 
\end{lemma}

\begin{proof}[Proof of Lemma~\ref{lem:1}]
Take any $y\in B_x(r\de)$. Then 
\begin{equation*}
	y=x+(y-x)=(1-\de)\frac x{1-\de}+\de\,\frac{y-x}\de\in K, 
\end{equation*}
because $\frac x{1-\de}\in K$, $\frac{y-x}\de\in B_0(r)\subseteq K$, and the set $K$ is convex. 
Thus, indeed $B_x(r\de)\subseteq K$. 
\end{proof}

Take any 
\begin{equation}\label{eq:ep}
	\ep\in(0,\pi/2). 
\end{equation}
Take any distinct points $x$ and $y$ in $\R^d$, and let 
\begin{equation}\label{eq:ell}
	\ell:=\|x-y\|. 
\end{equation}
Let $S_{\ep;x,y}$ denote the set of all points $z\in\R^d$ such that the angle $A_{xzy}$ at the vertex $z$ in the triangle $xyz$ is $>\pi-\ep$. In view of the condition $\ep\in(0,\pi/2)$, one may note that for any $z\in S_{\ep;x,y}$ the angle $A_{xzy}$ coincides with the largest angle (denoted here by $A_{\{x,y,z\}}$) in the triangle $xyz$. 
Let 
\begin{equation*}
	V_{\ep;d;x,y}:=\V_d(S_{\ep;x,y}). 
\end{equation*}
 
\begin{lemma}\label{lem:2}
One has 
\begin{equation*}
	V_{\ep;d;x,y}\sim\rho_d\,\ell^d\ep^{d-1}, 
\end{equation*}
where $\rho_d$ and $\ell$ are as in \eqref{eq:rho} and \eqref{eq:ell}. 
\end{lemma}

Here and in what follows, it is assumed by default that the limit relations are as $\ep\downarrow0$, unless specified otherwise.  
For any positive expressions $\EE_1$ and $\EE_2$, we write $\EE_1\sim\EE_2$ if $\EE_1/\EE_2\to1$, $\EE_1\O\EE_2$ if $\limsup(\EE_1/\EE_2)<\infty$, and $\EE_1\ll\EE_2$ if $\EE_1/\EE_2\to0$. 

\begin{proof}[Proof of Lemma~\ref{lem:2}]
By homothety, wlog $\ell=2$. 
Note that the angle $A_{xzy}$ equals $\pi-\ep$ if and only if the arc of the circle through the points $x,z,y$ (in this order) is such that the smaller one of the two adjacent angles between the line $xy$ (through the points $x,y$) and the line tangent to the mentioned arc at the point $x$ equals $\ep$ (of course, the roles of $x$ and $y$ are interchangeable here). It follows that $V_{\ep;d;x,y}$ equals twice the volume of the body of revolution obtained by the rotation in $\R^d$ of the curvilinear triangle $AEBCDA$ about the line $AC$, where the coordinates of the points $
A,B,C,D,E$ in some orthonormal basis of some two-dimensional subspace of $\R^d$ are as follows: 
\begin{gather*}
	A=(\cot\ep,1), 
	\quad B=(\csc\ep,0), 
	\quad C=(\cot\ep,0), 
	\\
	\quad D=(\cot\ep,\eta), 
	\quad E
	=\big(\sqrt{\csc^2\ep-\eta^2},\eta\big),
\end{gather*}
$0<\eta<1$, the curve $AEB$ is the arc of the circle centered at $O=(0,0)$, and the curvilinear triangle $AEBCDA$ is the convex hull of the union of the arc $AEB$ and the 
singleton set $\{C\}$. This is illustrated in the picture here. 



%
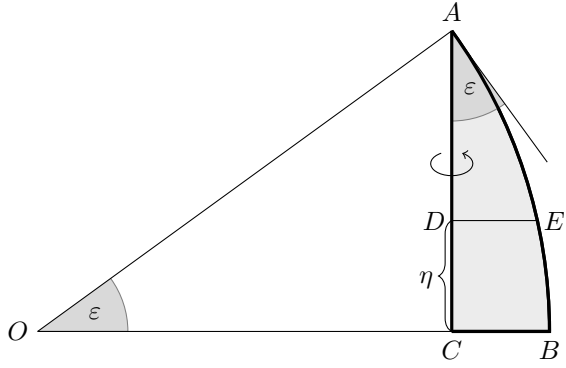
\begin{center}
\begin{tikzpicture}[scale=8]
\def\Ax{0.688191}
\def\A{(0.688191,0.5)}
\def\B{(0.850651,0)}
\def\C{(\Ax,0)}
\def\D{(\Ax, 0.1844)}
\def\E{(0.830424, 0.1844)}
\def\t1{30}
\colorlet{anglecolor}
{gray}
\colorlet{fillcolor}
{gray!30}
\filldraw[very thick,fill=gray!15,draw=black] \C -- (0.850651,0mm) arc(0:36:0.850651);
\filldraw[fill=fillcolor,draw=anglecolor] (0,0) -- (1.5mm,0pt) arc(0:36:1.5mm);
\draw (18:1mm) node
{$\ep$};
\filldraw[fill=
fillcolor,draw=
anglecolor] \A -- (0.688191,0.35) arc(270:306:0.15);
\draw (0.719093, 0.404894) node
{$\ep$};
\draw
(0,0) -- \B;
\draw[very thick](0.850651,0mm) arc(0:36:0.850651);
\draw[
very thick] 
\A -- (\Ax, 0);
\draw[very thick] \B -- \C;
\draw[thin] \D -- \E;
\draw (0,0) -- \A;
\draw[thin] \A -- 
(0.846893, 0.281565);
\draw[below,left] (0,0) node {$O$};
\draw[above] \A node {$A$};
\draw[below] \B node {$B$};
\draw[below] \C node {$C$};
\draw[left] (\Ax+.005, 0.1844) node {$D$};
\draw[left] (\Ax-.015, 0.085) node {$\eta$};
\draw[right] (0.830424-.005, 0.1844) node {$E$};
\draw[decorate,decoration=
{brace,amplitude=5pt}] \C -- \D;
\draw[decoration={brace}] \C -- \D;
\draw[->] (0.670691, 0.296401) arc (90+30:450-30:.035 and .02);
\end{tikzpicture}
\end{center}

Note that $|CD|=\eta$ and 
\begin{align*}
	|DE|=g(\eta):=&\sqrt{\csc^2\ep-\eta^2}-\cot\ep 
	=\frac{1-\eta^2}{\sqrt{\csc^2\ep-\eta^2}+\cot\ep} 
		\sim\tfrac12\,(1-\eta^2)\ep. 
\end{align*}
So, letting $\beta_{d-1}$ denote the volume of the unit ball in $\R^{d-1}$, we have 
\begin{multline*}
	V_{\ep;d;x,y}=2\int_0^1\beta_{d-1}\,g(\eta)^{d-1}\dd\eta
	\sim2\frac{\beta_{d-1}\ep^{d-1}}{2^{d-1}}\int_0^1(1-\eta^2)^{d-1}\dd\eta
	=\rho_d\,2^d\ep^{d-1}, 
\end{multline*}
because 
\begin{equation}\label{eq:be,...}
	\beta_{d-1}=\frac{\pi^{(d-1)/2}}{\Ga((d+1)/2)}\quad\text{and}\quad 
	\int_0^1(1-\eta^2)^{d-1}\dd\eta=\frac{\sqrt\pi\,\Ga(d)}{2\Ga(d+1/2)}.   
\end{equation}
Thus, Lemma~\ref{lem:2} is proved.  
\end{proof}

\begin{lemma}\label{lem:3}
For the conditional probability that $X_{\{1,2,3\}}>\pi-\ep$ given $P_1$ and $P_2$, one has  
\begin{equation}\label{eq:cond,X}
\P(X_{\{1,2,3\}}>\pi-\ep|P_1,P_2)\O\ep^{d-1}.  
\end{equation} 
\end{lemma} 

\begin{proof}[Proof of Lemma~\ref{lem:3}]
The event 
$X_{\{1,2,3\}}>\pi-\ep$ implies that 
\begin{equation}\label{eq:dist}
	\text{the distance from $P_3$ to the line $P_1P_2$ is $\le(|P_1P_3|\wedge|P_2P_3|)\sin\ep\le\ep$,}
\end{equation}
since the diameter of the set $K$ was assumed to equal $1$; so, then the point $P_3$ lies in the intersection (say $I$) of $K$ with (
the convex hull of)
the round cylinder of radius $\ep$ (with axis $P_1P_2$). The intersection of $I$ with any line parallel to $P_1P_2$ is a segment of length not exceeding the diameter of $K$, which is $1$. So, $\V_d(I)\O\ep^{d-1}$, and hence 
\eqref{eq:cond,X} follows. 
\end{proof}

\begin{lemma}\label{lem:4}
For any $t$ and $s$ in $T_n$ such that $t\ne s$ and $t\cap s\ne\emptyset$, one has  
\begin{equation*}
\P(X_t>\pi-\ep,X_s>\pi-\ep)\O\ep^{2(d-1)}, 
\end{equation*}
uniformly over all such $s$ and $t$.  
\end{lemma} 

\begin{proof}[Proof of Lemma~\ref{lem:4}] Take any $t$ and $s$ in $T_n$ such that $t\ne s$ and $t\cap s\ne\emptyset$. Then the cardinality of $t\cap s$ is either $2$ or $1$. 
So, wlog one of the following two cases holds: (i) $t=\{1,2,3\}$ and $s=\{1,2,4\}$ or (ii) $t=\{1,2,3\}$ and $s=\{1,4,5\}$. 
In case (i), 
\begin{align*}
	\P(X_t>\pi-\ep,X_s>\pi-\ep)&=\E\ii{X_{\{1,2,3\}}>\pi-\ep}\P(X_{\{1,2,4\}}>\pi-\ep|P_1,P_2,P_3) \\ 
	&=\E\ii{X_{\{1,2,3\}}>\pi-\ep}\P(X_{\{1,2,4\}}>\pi-\ep|P_1,P_2) \\ 
	&\O\E\ii{X_{\{1,2,3\}}>\pi-\ep}\ep^{d-1} \\ 
	&=\E\P(X_{\{1,2,3\}}>\pi-\ep|P_1,P_2)\ep^{d-1}\O\ep^{2(d-1)},  
\end{align*}
by Lemma~\ref{lem:3}; here and elsewhere, $\ii{\cdot}$ is the indicator function. 
Similarly, in case~(ii), 
\begin{align*}
	\P(X_t>\pi-\ep,X_s>\pi-\ep)&=\E\ii{X_{\{1,2,3\}}>\pi-\ep}\P(X_{\{1,4,5\}}>\pi-\ep|P_1,P_2,P_3) \\ 
	&=\E\ii{X_{\{1,2,3\}}>\pi-\ep}\P(X_{\{1,4,5\}}>\pi-\ep|P_1) \\ 
	&=\E\ii{X_{\{1,2,3\}}>\pi-\ep}\,\E\big(\P(X_{\{1,4,5\}}>\pi-\ep|P_1,P_4)|P_1\big) \\ 
	&\O\E\ii{X_{\{1,2,3\}}>\pi-\ep}\ep^{d-1} \O\ep^{2(d-1)}.   
\end{align*}
Lemma~\ref{lem:4} is now proved. 
\end{proof}

\begin{lemma}\label{lem:5}
One has  
\begin{equation}\label{eq:P(X>pi-ep)}
\P(X_{\{1,2,3\}}>\pi-\ep)\sim\la_d(K)\ep^{d-1}.  
\end{equation}   
\end{lemma} 

\begin{proof}[Proof of Lemma~\ref{lem:5}] 
The event $\{X_{\{1,2,3\}}>\pi-\ep\}$ is the union of the events $\{A_{P_1P_3P_2}>\pi-\ep\}$, $\{A_{P_1P_2P_3}>\pi-\ep\}$, and $\{A_{P_2P_1P_3}>\pi-\ep\}$, and these events are pairwise disjoint, in view of \eqref{eq:ep}. So, 
\begin{equation}\label{eq:=3 times}
	\P(X_{\{1,2,3\}}>\pi-\ep)=3\P(A_{P_1P_3P_2}>\pi-\ep). 
\end{equation}
By Lemma~\ref{lem:1} with 
\begin{equation*}
	\de:=\ep/r, 
\end{equation*}
the event $\big\{\{P_1,P_2\}\subseteq K_\de\big\}$ implies $B_{P_1}(\ep)\cup B_{P_2}(\ep)\subseteq K$. 
On the other hand, in view of \eqref{eq:ep}, the event $\big\{A_{P_1P_3P_2}>\pi-\ep\big\}$ implies that the point on the line $P_1P_2$ that is the closest one to the $P_3$ is between $P_1$ and $P_2$. 
Hence, in view \eqref{eq:dist}, the event $\big\{A_{P_1P_3P_2}>\pi-\ep,\{P_1,P_2\}\subseteq K_\de\big\}$ implies that $P_3$ is in the convex hull of the set $B_{P_1}(\ep)\cup B_{P_2}(\ep)$ and thus in $K$. 
So, on the event $\big\{\{P_1,P_2\}\subseteq K_\de\big\}$, 
\begin{equation*}
	\P(A_{P_1P_3P_2}>\pi-\ep|P_1,P_2)=\frac{V_{\ep;d;P_1,P_2}}{\V_d(K)}
	\sim\rho_d\,\frac{|P_1P_2|^d}{\V_d(K)}\ep^{d-1}  
\end{equation*}
by Lemma~\ref{lem:2}, 
whence 
\begin{multline}\label{eq:P1,P2 in}
	\P(A_{P_1P_3P_2}>\pi-\ep,\{P_1,P_2\}\subseteq K_\de)
	\sim\rho_d\,\frac{\ep^{d-1}}{\V_d(K)}\,\E|P_1P_2|^d\,\ii{\{P_1,P_2\}\subseteq K_\de} \\ 
	=\rho_d\,\frac{\ep^{d-1}}{\V_d(K)}\,(1-\de)^{3d}\E|P_1P_2|^d
	\sim\rho_d\,\frac{\ep^{d-1}}{\V_d(K)}\,\E|P_1P_2|^d;  
\end{multline}
here are details concerning the scaling factor $(1-\de)^{3d}$ in the above display: for any $t\in(0,1)$, 
\begin{align*}
 \E|P_1P_2|^d\,\ii{\{P_1,P_2\}\subseteq tK}
 =&\int_{tK}\int_{tK}\|x_1-x_2\|^d\frac{\V_d(\dd x_1)}{\V_d(K)}\,\frac{\V_d(\dd x_2)}{\V_d(K)} \\ 
 =&\int_{K}\int_{K}\|ty_1-ty_2\|^d\frac{\V_d(t\dd y_1)}{\V_d(K)}\,\frac{\V_d(t\dd y_2)}{\V_d(K)} \\ 
 =&t^{3d}\int_{K}\int_{K}\|y_1-y_2\|^d\frac{\V_d(\dd y_1)}{\V_d(K)}\,\frac{\V_d(\dd y_2)}{\V_d(K)}
 =t^{3d}\E|P_1P_2|^d. 
\end{align*}

Note next that $\P(P_1\notin K_\de)=1-(1-\de)^d\ll1$. On the other hand, by Lemma~\ref{lem:3}, $\P(A_{P_1P_3P_2}>\pi-\ep|P_1)\O\ep^{d-1}$. So, 
\begin{equation}\label{eq:P1,P2 out}
\begin{aligned}
	&\P(A_{P_1P_3P_2}>\pi-\ep,P_2\notin K_\de) \\ 
	=&\P(A_{P_1P_3P_2}>\pi-\ep,P_1\notin K_\de)
	\O\ep^{d-1}\P(P_1\notin K_\de)\ll\ep^{d-1}. 
\end{aligned}	
\end{equation}
Now Lemma~\ref{lem:5} immediately follows by \eqref{eq:=3 times}, \eqref{eq:P1,P2 in}, \eqref{eq:P1,P2 out}, and \eqref{eq:la}. 
\end{proof}

To complete the proof of Theorem~\ref{th:}, we shall use the key 
result of Galambos~\cite{galambos72}, which is in turn based on a combinatorial graph sieve theorem due to R\'enyi \cite{renyi61}. For readers' convenience, let us restate here the result of \cite{galambos72}, as follows. 

For each natural $n$, let $T_n$ be a set of cardinality $
|T_n|\to\infty$ and let $E_n$ be a set of subsets of $T_n$ of cardinality $2$ such that 
\begin{equation}\label{eq:|E_n|<<}
	|E_n|\ll|T_n|^2. 
\end{equation}
In this setting, all limit relations are stated for $n\to\infty$. 

For natural $n$ and $k$ and for $\al\in\{0,1\}$, let $H_{n,k}^{(\al)}$ denote the set of all subsets $F$ of $T_n$ of cardinality $|F|=k$ such that exactly $\al$ subsets of $F$ of cardinality $2$ belong to $E_n$.  

Take any real number $a>0$, and let a sequence of real numbers $c_n=c_n(a)$ be such that the following conditions hold: 
\begin{gather}
	\sum_{t\in T_n}\P(X_t\ge c_n)\longrightarrow a; \label{eq:to a} \\
	  \sup_{n,\,t\in T_n}|T_n|\P(X_t\ge c_n)<\infty; \label{eq:sup<} \\ 
	  |E_n|\max_{\{t,s\}\in E_n}\P(X_t\ge c_n,X_s\ge c_n)\longrightarrow0; \label{eq:|E_n|max<<}\\ 
\intertext{for each natural $k$ }	  
	  \sum_{F\in H_{n,k}^{(0)}}\Big[\P\Big(\min_{t\in F}X_t\ge c_n\Big)
	  -\prod_{t\in F}\P(X_t\ge c_n)\Big]\longrightarrow0; \label{eq:H0} \\
\intertext{and for each natural $k$ there is a real number $d_k$ such that for all $F\in H_{n,k}^{(1)}$ }
\P\Big(\min_{t\in F}X_t\ge c_n\Big)
\le d_k \P(X_s\ge c_n,X_r\ge c_n)\prod_{t\in F\setminus\{s,r\}}\P(X_t\ge c_n), \label{eq:H1}	  
\end{gather}
where $\{s,r\}$ is the only subset of $F$ of cardinality $2$ that belongs to $E_n$. 

Then 
\begin{equation}\label{eq:gal}
	\P\Big(\max_{t\in T_n}X_t<c_n\Big)\longrightarrow e^{-a}. 
\end{equation}
\big(The theorem in \cite{galambos72} was stated in terms of (sub)sequences rather than subsets, but the formulation given above is easily seen to be equivalent to that in \cite{galambos72}.\big)  

The four conditions \eqref{eq:|E_n|<<}, \eqref{eq:|E_n|max<<}, \eqref{eq:H0}, and \eqref{eq:H1} specify a notion of weak dependence of the r.v.'s $X_t$. In particular, in the case when the $X_t$'s are independent, it is easy to see that none of these four conditions is needed to deduce \eqref{eq:gal} already from \eqref{eq:to a} and \eqref{eq:sup<}.  

Now we are ready to complete the proof of Theorem~\ref{th:}. Indeed, let $T_n$ and $X_t$ be as described in the beginning of Section~\ref{intro}, so that 
\begin{equation}\label{eq:|T_n|}
	|T_n|=\binom n3\sim\frac{n^3}{3!}  
\end{equation}
and the r.v.'s $X_t$ are exchangeable. 
Next, 
let $E_n$ be the set of subsets $\{t,s\}$ of $T_n$ of cardinality $2$ such that $t\cap s\ne\emptyset$. 

Then $|E_n|\O n^3\cdot n^2$, so that condition \eqref{eq:|E_n|<<} holds. 

Take now indeed any real $a>0$ and let  
\begin{equation}\label{eq:c_n}
	c_n:=c_n(a):=\pi-\ep_n,\quad\text{where}\quad
	\ep_n:=\Big(\frac a{|T_n|\la_d(K)}\Big)^{1/(d-1)},   
\end{equation}
so that $c_n\uparrow\pi$; from now on, all limit relations are stated for $n\to\infty$.  
Then conditions \eqref{eq:to a} and \eqref{eq:sup<} hold by Lemma~\ref{lem:5}. 
Condition \eqref{eq:|E_n|max<<} follows by Lemma~\ref{lem:4} and \eqref{eq:|E_n|<<}. 
Condition \eqref{eq:H0} is trivial here, because for each $F\in H_{n,k}^{(0)}$ the family of r.v.'s $(X_t)_{t\in F}$ is independent. 
Finally, condition \eqref{eq:H1} holds (with $d_k=1$) because for each $F\in H_{n,k}^{(1)}$ and $s,r$ as described in that condition, the family of r.v.'s $(X_t)_{t\in F\setminus\{s,r\}}$ is independent in itself and also independent of the random pair $(X_s,X_r)$. 

Thus, the conclusion \eqref{eq:gal} holds, with $c_n$ as in \eqref{eq:c_n}. In view of \eqref{eq:Y} and \eqref{eq:phi}, 
\begin{equation*}
\P\Big(\max_{t\in T_n}X_t<c_n\Big)=\P\Big(Y_{n,d}\frac{|T_n|}{n^3/3!}>a\Big).  	
\end{equation*} 
So, to finish the proof of Theorem~\ref{th:}, it remains to recall \eqref{eq:|T_n|}. 
\qed

\appendix 

\section{Proof of Proposition~\ref{prop:B,min elong}}\label{min elong}


As mentioned before, this proof is based on the Steiner symmetrization. 

Let $E_d(K)$ denote $\E|P_1P_2|^d$, where, as before, $P_1$ and $P_2$ are random points drawn independently and uniformly from $K$. Let $\bar K$ denote the closure of $K$. Then for any real $\de>0$ one has $K\subseteq \bar K\subseteq (1+\de)K$, whence 
\begin{equation*}
E_d(K)\le E_d(\bar K)\le E_d ((1+\de)K)=
(1+\de)^dE_d(K)\underset{\de\downarrow0}\longrightarrow E_d(K), 
\end{equation*}
so that $E_d(\bar K)=E_d(K)$. Similarly, $\V_d(\bar K)=\V_d(K)$. 
So, wlog the bounded convex set $K$ 
is closed and hence compact, which will be assumed henceforth. 

The Steiner symmetrization can be described as follows. Take any unit vector $u\in\R^d$. For each $x\in\R^d$, there are uniquely determined $\th(x)=\th_u(x)$ in $\R$ and $H(x)=H_u(x)$ in the orthogonal complement $\uperp:=\{y\in\R^d\colon\langle y,u\rangle=0\}$ of the singleton set $\{u\}$ to $\R$ such that 
\begin{equation}\label{eq:th,H}
	x=\th(x)u+H(x). 
\end{equation}
For each $y\in\uperp$, let 
\begin{equation*}
	s(y):=s_{K,u}(y):=\{\al\in\R\colon\al u+y\in K\}. 
\end{equation*}
Let then 
\begin{equation*}
		Y_u:=Y_{K,u}:=\big\{y\in\uperp\colon s(y)\ne\emptyset\big\}. 
\end{equation*}
Clearly, for each $y\in\uperp$, the set $s(y)$ is a compact convex subset of $\R$, and this set is nonempty if $y\in\Y_u$. So, for each $y\in\Y_u$ there exist unique real numbers $a(y)=a_K(y)$ and $b(y)=b_K(y)$ such that $a(y)\le b(y)$ and 
\begin{equation*}
	s(y)=[a(y),b(y)]. 
\end{equation*}
Then the Steiner symmetrization, say $\S_uK$, of $K$ along the unit vector $u$ is the set  
\begin{equation*}
	\S_uK:=\bigcup_{y\in Y_{K,u}} \Big\{\al u+y\colon \frac{a(y)-b(y)}2\le\al\le\frac{b(y)-a(y)}2\Big\}.  
\end{equation*}
It is well known and easy to see that $\S_uK$ is a compact convex set, of the same volume as $K$: 
\begin{equation*}
	\V_d(\S_uK)=\V_d(K). 
\end{equation*}
One may also note that 
that the compact convex set $\S_uK$ is determined by the conditions 
\begin{equation*}
	Y_{\S_uK,u}=Y_{K,u}\quad\text{and}\quad s_{\S_uK,u}(y)=\Big[\frac{a(y)-b(y)}2,\frac{b(y)-a(y)}2\Big]
\end{equation*}
for all $y\in Y_{K,u}$. 

Now we are ready to state 
\begin{lemma}\label{lem:steiner ineq}
$E_d(\S_uK)\le E_d(K)$. 
Moreover, $E_d(\S_uK)=E_d(K)$ if and only if $K$ is a ball. 
\end{lemma}

\begin{proof}[Proof of Lemma~\ref{lem:steiner ineq}]
For the random points $P_1$ and $P_2$ as before and $j\in\{1,2\}$, in view of \eqref{eq:th,H} one can write 
\begin{equation*}
	P_j=\xi_j u+Q_j,\quad\text{where}\ \xi_j:=\th(P_j) \text{ and } Q_j:=H(P_j). 
\end{equation*}
Clearly, the random pairs $(\xi_1,Q_1)$ and $(\xi_2,Q_2)$ are independent copies of each other. Also, conditionally on $(Q_1,Q_2)$, the r.v.'s $\xi_1$ and $\xi_2$ are independent, and 
for each $j\in\{1,2\}$ the conditional distribution of $\xi_j$ given $(Q_1,Q_2)$ is uniform over the interval $[A_j,B_j]:=[a(Q_j),b(Q_j)]$. 

\rule{0pt}{12pt}Let now $\tP_j$, $\txi_j$, $\tQ_j$, $\tA_j$, $\tB_j$ be defined similarly to 
$P_j$, $\xi_j$, $Q_j$, $A_j$, $B_j$ (respectively), but with $\S_uK$ in place of $K$. 
Note that the random pairs $(Q_1,Q_2)$ and $(\tQ_1,\tQ_2)$ are the same in distribution, and we may and will assume that they are just equal to each other: $(\tQ_1,\tQ_2)=(Q_1,Q_2)$. 
Then, given $(Q_1,Q_2)$, for each $j\in\{1,2\}$ the conditional distribution of $\xi_j$ is the same as that of $\txi_j+\frac{A_j+B_j}2$. So, letting 
\begin{equation*}
	\eta:=\xi_1-\xi_2,\quad
	\teta:=\txi_1-\txi_2,\quad C:=\tfrac{A_1+B_1}2-\tfrac{A_2+B_2}2, 
\end{equation*}
we see that, 
given $(Q_1,Q_2)$, the conditional distribution of $\eta$ is the same as that of $\teta+C$, whereas the conditional distribution of the real valued r.v.\ $\teta$ given $(Q_1,Q_2)$ is symmetric (about $0$). 
%
%
So, introducing $Z:=|Q_1Q_2|=\|Q_1-Q_2\|$ and writing the instance $|P_1P_2|^2=\eta^2+Z^2$ of the Pythagoras theorem, 
we have 
\begin{equation}\label{eq:E=,E=}
	E_d(K)=\E g_{\teta,Z}(|C|)\quad\text{and}\quad E_d(\S_uK)=\E g_{\teta,Z}(0), 
\end{equation}
where 
\begin{equation*}
	g_{t,z}(c):=g_{t,z,d}(c):=\tfrac12\big((t+c)^2+z^2)^{d/2}+\tfrac12\big((-t+c)^2+z^2)^{d/2}. 
\end{equation*}
for real $t,z,c$. 
For each $(t,z,d)\in\R\times(0,\infty)\times[1,\infty)$, the function $g_{t,z}=g_{t,z,d}$ is even and strictly convex on $\R$, and hence strictly increasing on $[0,\infty)$. 
In view of \eqref{eq:E=,E=}, this immediately yields the inequality $E_d(\S_uK)\le E_d(K)$. 

Moreover, it follows that the equality $E_d(\S_uK)=E_d(K)$ is possible only if $C=0$ almost surely. 
By the Fubini theorem, this implies that there is some $y_1\in Y_u$ such that for almost all $y\in Y_u$ one has 
\begin{equation}\label{eq:a+b}
	\frac{a(y)+b(y)}2=
	\tau_u:=\frac{a(y_1)+b(y_1)}2;      
\end{equation}
the dependence of the functions $a$ and $b$ on the unit vector $u$ (and on the set $K$) is implicit here. 

Let us now show that \eqref{eq:a+b} holds for all $y\in Y_u$. Toward this end, note first that the set $Y_u$ is the orthogonal projection of $K$ onto $\uperp$, and so, $Y_u$ is convex and compact. Moreover, for any $y_0$ and $y_1$ in $Y_u$ and any $t\in(0,1)$ one has $b(y_j)u+y_j\in K$ for $j\in\{1,2\}$, whence $$[(1-t)b(y_0)+tb(y_1)]u+(1-t)y_0+ty_1=(1-t)(b(y_1)u+y_1)+t(b(y_2)u+y_2)\in K,$$ 
which yields $(1-t)b(y_0)+tb(y_1)\in s((1-t)y_0+ty_1)$,  
so that $(1-t)b(y_0)+tb(y_1)\le b((1-t)y_0+ty_1)$. This shows that the function $b$ is concave, on $Y_u$, and hence continuous on the interior $\inter Y_u$ of $Y_u$, by a well-known theorem (see e.g.\ \cite[Theorem~10.1]{rocka}). 
It follows that \eqref{eq:a+b} holds for all $y\in\inter Y_u$. 

Next, take any $y\in Y_u$. Take then $x:=b(y)u+y$, so that $x\in K$. By Lemma~\ref{lem:1}, for all $\de\in(0,1)$ one has $\inter K\ni(1-\de)x=(1-\de)b(y)u+(1-\de)y$, so that $(1-\de)y$ is in the orthogonal projection of $\inter K$ onto $\uperp$ and hence in $\inter Y_u$. It follows that 
\begin{equation*}
	\frac{a((1-\de)y)+b((1-\de)y)}2=\tau_u, 
\end{equation*}
for all $\de\in(0,1)$. 
By the compactness of $K$, there are real $a_*$ and $b_*$ and a sequence $(\de_m)$ in $(0,1)$ converging to $0$ such that $b((1-\de_m)y)\to b_*$ and $a((1-\de_m)y)\to a_*$, whence 
$K\ni b((1-\de_m)y)u+(1-\de_m)y\to b_*u+y$. Therefore and because $K$ is closed, we have $b_*u+y\in K$, and so, $b_*\le b(y)$. On the other hand, by the concavity of $b$, $b((1-\de_m)y)\ge(1-\de_m)b(y)+\de_m b(0)\to b(y)$, so that $b_*=\lim_m b((1-\de_m)y)\ge b(y)$. 
Thus, $b(y)=b_*=\lim_m b((1-\de_m)y)$ and, similarly, $a(y)=\lim_m a((1-\de_m)y)$. 
Since $(1-\de_m)y\in\inter Y_u$ and \eqref{eq:a+b} was established for all $y\in\inter Y_u$, it follows that indeed \eqref{eq:a+b} holds for any $y\in Y_u$. 

This means that the set $K$ is symmertic about the plane 
$$\Pi_u:=\{x\in\R^d\colon \langle x,u\rangle=\tau_u\},$$ 
for any unit vector $u\in\R^d$. By translation, wlog $\bigcap_{i=1}^d\Pi_{e_i}=\{0\}$, where $(e_1,\dots,e_d)$ is (say) the standard basis of $\R^d$. Since $K$ is symmertic about each of the ``coordinate'' hyperplanes $\Pi_{e_1},\dots,\Pi_{e_d}$, it is easy to see that $K$ is centrally symmetric about the origin. 

So, taking any unit vector $v\in\R^d$, one has $a(v)+b(v)=0$. It follows that the hyperplane $\Pi_v$ of symmetry of $K$ must pass through the origin. Thus, $K$ is invariant with respect to the reflection in any hyperplane through the origin. On the other hand, by the Cartan--Dieudonn\'e theorem (see e.g.\ \cite{fuller}), 
any orthogonal transformation is the composition of reflections. We conclude that the compact convex set $K$ is invariant with respect to any orthogonal transformation; hence, $K$ is a ball. 

To complete the proof of Lemma~\ref{lem:steiner ineq}, it remains to note that, if $K$ is a ball, then obviously $E_d(\S_uK)=E_d(K)$.  
\end{proof}

\begin{lemma}\label{lem:ineq}
Let $B$ denote the unit ball in $\R^d$, and suppose that $\V_d(K)=\V_d(B)$. Then 
$E_d(K)\ge E_d(B)$. 
\end{lemma}

\begin{proof}[Proof of Lemma~\ref{lem:ineq}]
Take any $\de\in(0,1)$. By a well-known result (see cf.\ \cite[Theorem~1.5]{klartag04}), there exist a natural $m$ and unit vectors $u_1,\dots,u_m$ in $\R^d$ such that 
\begin{equation*}
	(1+\de)B\supseteq K_m:=\S_{u_m}\cdots\S_{u_1}K\supseteq(1-\de)B. 
\end{equation*}
So, in view of Lemma~\ref{lem:steiner ineq}, 
\begin{align*}
E_d(K)\ge E_d(K_m)&\ge E_d((1-\de)B)\frac{\V_d((1-\de)B)}{\V_d(K_m)} \\ 
&\ge E_d((1-\de)B)\frac{\V_d((1-\de)B)}{\V_d((1+\de)B)}
=\frac{(1-\de)^{2d}}{(1+\de)^d} E_d(B),
\end{align*}
for any $\de\in(0,1)$. Letting now $\de\downarrow0$, one immediately obtains Lemma~\ref{lem:ineq}. 
\end{proof}

Now it is easy to finish the proof of Proposition~\ref{prop:B,min elong}. Indeed, by the scaling properties of $E_d(K)=\E|P_1P_2|^d$ and $\V_d(K)$, wlog $\V_d(K)=\V_d(B)$, where $B$ is the unit ball in $\R^d$, as in 
Lemma~\ref{lem:ineq}. From that lemma and the definition \eqref{eq:la} of $\la_d(K)$, it follows immediately that $\la_d(K)\ge\la_d(B)$. 

Suppose now that $\la_d(K)=\la_d(B)$. Then $E_d(K)=E_d(B)$. On the other hand, 
$E_d(K)\ge E_d(\S_uK)$ by Lemma~\ref{lem:steiner ineq}, and $E_d(\S_uK)\ge E_d(B)$ by Lemma~\ref{lem:ineq}, since $\V_d(\S_uK)=\V_d(K)=\V_d(B)$. Thus, $E_d(K)\ge E_d(\S_uK)\ge E_d(B)=E_d(K)$, whence $E_d(\S_uK)=E_d(K)$, so that, again by Lemma~\ref{lem:steiner ineq}, $K$ is a ball. 
This completes the proof of Proposition~\ref{prop:B,min elong}.

\section{Proof of Proposition~\ref{lem:la(B)}}\label{elong}

%

The random pair $(P_1,P_2)$ equals $(R_1U_1,R_2U_2)$ in distribution, where $R_1$ and $R_2$ are real-valued r.v.'s each with density $\R\ni r\mapsto d\,r^{d-1}\ii{0<r<1}$, $U_1$ and $U_2$ are random vectors each uniformly distributed on the unit sphere, and $R_1,U_1,R_2,U_2$ are independent. So, letting $T$ denote the cosine of the angle between the random vectors $U_1$ and $U_2$, we see that the r.v.'s $R_1,R_2,T$ are independent, and 
$|P_1P_2|^2$ equals $R_1^2+R_2^2-2R_1R_2T$ in distribution. Moreover -- because, by the spherical symmetry, $T$ is independent of (say) $U_1$ -- the density of $T$ is $\R\ni t\mapsto C_d\,(1-t^2)^{(d-3)/2}\ii{-1<t<1}$, where 
\begin{equation}\label{eq:C_d}
	C_d:=\frac{\Ga(d/2)}{\sqrt\pi\,\Ga((d-1)/2)}\sim\sqrt{\frac d{2\pi}};  
\end{equation}
all the limit relations in this proof are of course for $d\to\infty$.  
Hence,
\begin{equation}\label{eq:E=}
	\E|P_1P_2|^d=C_d\,d^2 (J_{d,1}+J_{d,2}), 
\end{equation}
where 
\begin{gather*}
J_{d,1}:= \int_{-1}^0\Psi_d(t)(1-t^2)^{(d-3)/2}\dd t,\quad 
J_{d,2}:= \int_0^1\Psi_d(t)(1-t^2)^{(d-3)/2}\dd t, \\ 
	\Psi_d(t):=\int_0^1\dd r_1\int_0^1\dd r_2\;r_1^{d-1}r_2^{d-1}(r_1^2+r_2^2-2r_1r_2t)^{d/2}. 
\end{gather*}

For $t>0$, the integrand in the integral expression for $\Psi_d(t)$ is no greater than $2^{d/2}$, and so, $\Psi_d(t)\le2^{d/2}$. Therefore, for $d\ge3$ 
\begin{equation}\label{eq:J2<}
	J_{d,2}\le2^{d/2}. 
\end{equation}


Let us now estimate $J_{d,1}$. Note that 
\begin{equation}\label{eq:J1,J11,J12}
	J_{d,1,1}\le J_{d,1}\le J_{d,1,1}+2J_{d,1,2}, 
\end{equation}
where 
\begin{align*}
	J_{d,1,j}:=&\int_{-1}^0\Psi_{d,j}(t)(1-t^2)^{(d-3)/2}\dd t, \\ 
	\Psi_{d,1}(t):=&\int_{1-\de}^1\dd r_1\int_{1-\de}^1\dd r_2\;r_1^{d-1}r_2^{d-1}(r_1^2+r_2^2-2r_1r_2t)^{d/2},  \\ 
	\Psi_{d,2}(t):=&\int_0^{1-\de}\dd r_1\int_0^1\dd r_2\;r_1^{d-1}r_2^{d-1}(r_1^2+r_2^2-2r_1r_2t)^{d/2},  \\ 
	\de:=&\de_d:=1/\sqrt{d-1}. 
\end{align*}

For $(r_1,r_2,t)\in
(0,1-\de)\times(0,1)\times(-1,0)$, the integrand $r_1^{d-1}r_2^{d-1}(r_1^2+r_2^2-2r_1r_2t)^{d/2}$ in $\Psi_{d,2}(t)$ does not exceed 
\begin{align*}
&(1-\de)^{d-1}\big((1-\de)^2+1-2(1-\de)t\big)^{d/2}\\ 
&\le (1-\de)^{d-1}(1-t)^{d/2}2^{d/2}\le2^{d/2}e^{-\sqrt{d-1}}(1-t)^{d/2}, 	
\end{align*}
and so, the latter expression is also an upper bound on $\Psi_{d,2}(t)$. It follows that 
\begin{equation}\label{eq:J12<}
	J_{d,1,2}\le 2^{d/2}e^{-\sqrt{d-1}}Q_d, 
\end{equation}
where 
\begin{equation*}
	Q_d:=\int_{-1}^0(1-t)^{d/2}(1-t^2)^{(d-3)/2}\dd t. 
\end{equation*}

Letting $v:=1-r_1$ and $w:=1-r_2$, we have 
\begin{align*}
	r_1^2r_2^2(r_1^2+r_2^2-2r_1r_2t)=&2(1-t)-6(1-t)(v+w)+O(v^2+w^2) \\ 
	=&2(1-t)[1-3(v+w)+O(v^2+w^2)] \\  
	=&2(1-t)\exp\{-3(v+w)+O(v^2+w^2)\} \\    
	=&2(1-t)\exp\{-(3+o(1))(v+w)\}    
\end{align*}
for $(r_1,r_2,t)\in(1-\de,1)\times(1-\de,1)\times(-1,0)$,  
whence 
\begin{align*}
	\Psi_{d,1}(t)\sim & 
	\,2^{(d-1)/2}(1-t)^{(d-1)/2}(1^2+1^2-2\times1\times1\times t)^{1/2} \\ 
	&\times
	\Big(\int_0^\de\dd v\;\exp\{-(3+o(1))v(d-1)/2\}\Big)^2 \\ 
	&\sim \,2^{d/2}(1-t)^{d/2}\frac4{9d^2}
\end{align*}
and 
\begin{equation*}
	J_{d,1,1}\sim 2^{d/2}\frac4{9d^2}\,Q_d.    
\end{equation*}
Comparing this with \eqref{eq:J12<}, we see that $J_{d,1,2}\ll J_{d,1,1}$. Recalling now \eqref{eq:J1,J11,J12}, we conclude that 
\begin{equation}\label{eq:J1 sim}
	J_{d,1}\sim J_{d,1,1}\sim 2^{d/2}\frac4{9d^2}\,Q_d. 
\end{equation}

The needed estimation of $Q_d$ is straightforward. Indeed, 
\begin{equation}\label{eq:Q_d=}
	Q_d=\int_{-1}^0(1-t)^{3/2}\exp\Big\{\frac{d-3}2\,\psi(t)\Big\}\dd t, 
\end{equation}
where $\psi(t):=2\ln(1-t)+\ln(1+t)$. One has $\psi'(t)=-\frac{1+3t}{1-t^2}$; $\psi''(t)=-\frac1{(1+t)^2}-\frac2{(1-t)^2}<0$; $\psi'(t)=0\iff t=t_*:=-\frac13$; $\exp\{\psi(t_*)\}=\frac{32}{27}$; and $\psi''(t_*)=-\frac{27}8$. 
So, standard reasoning yields 
\begin{align*}
	Q_d&\sim \int_{-1}^0(1-t_*)^{3/2}
	\exp\Big\{\frac{d-3}2\,\Big[\psi(t_*)+\frac12\psi''(t_*)(t-t_*)^2\Big]\Big\}\dd t \\ 
	&\sim (1-t_*)^{3/2}\exp\Big\{\frac{d-3}2\,\psi(t_*)\Big\}\frac{\sqrt{2\pi}}{\sqrt{-\psi''(t_*)d/2}}
	=\sqrt\pi\,\frac{3\sqrt3}{4\sqrt{d}}\,\Big(\frac{32}{27}\Big)^{d/2}.  
\end{align*}
Hence, by \eqref{eq:J1 sim},
\begin{equation*}
	J_{d,1}\sim 2^{d/2}\frac4{9d^2}\,\sqrt\pi\,\frac{3\sqrt3}{4\sqrt{d}}\,\Big(\frac{32}{27}\Big)^{d/2}
	=\frac{\sqrt\pi\,}{d^2\sqrt{3d}}\,\Big(\frac{64}{27}\Big)^{d/2}. 
\end{equation*} 
Comparing this with \eqref{eq:J2<}, we see that $J_{d,2}\ll J_{d,1}$. Now 
\eqref{eq:E=} and \eqref{eq:C_d} yield 
\begin{equation*}
	\E|P_1P_2|^d\sim C_d\,d^2 J_{d,1}
	\sim C_d\frac{\sqrt\pi\,}{\sqrt{3d}}\,\Big(\frac{64}{27}\Big)^{d/2}
		\sim
		\frac1{\sqrt{6}}\,\Big(\frac{8}{3\sqrt3}\Big)^d,   
\end{equation*}
which proves the first asymptotic equivalence in \eqref{eq:la(B)}. 
The second asymptotic equivalence there now follows immediately by \eqref{eq:la}, \eqref{eq:rho}, the first equality in \eqref{eq:be,...} (with $d-1$ replaced by $d$), and the asymptotic equivalence $\Ga(\al+1/2)\sim\sqrt\al\,\Ga(\al)$ as $\al\to\infty$.  \qed 





\bibliographystyle{abbrv}

\bibliography{C:/Users/ipinelis/Dropbox/mtu/bib_files/citations12.13.12}


\end{document}